\documentclass{article}
\usepackage[margin=1in, paperwidth=9.0in, paperheight=12in]{geometry}
\usepackage{amsmath,amsfonts,amsthm}
\usepackage{graphicx,epstopdf,float,subcaption}
\usepackage{enumerate}
\usepackage{mathtools}
\usepackage{array}
\newtheorem{thm}{Theorem}

\theoremstyle{definition}
\newtheorem{defn}{Definition}

\begin{document}
	\begin{center}
	{\huge Analysis of Fractional Ordered incommensurate Quadratic Jerk System}
	
	{\Large Rasika Deshpande$^{a}$$^\dagger$, Amey Deshpande$^{b}$ \\
		\small $^{a}$Department of Mathematics and Statistics, Dr. Vishwanath Karad MIT World Peace University, Pune,India, Email: rasikadeshpande933@gmail.com \\
		\small $^{b}$Department of Mathematics and Statistics, Dr. Vishwanath Karad MIT World Peace University, Pune, India, Email: 2009asdeshpande@gmail.com, amey.deshpande@mitwpu.edu.in \\
		\small $^{\dagger}$: corresponding author\\
	}
	
\end{center}	
	
	\begin{abstract}
		Fractional ordered dynamical systems (FODS) are being studied in the present due to their innate qualitative and quantitative properties and their applications in various fields. The Jerk system, which is a system involving three differential equations with quadratic complexity, arises naturally in wide ranging fields, and hence a qualitative study of solutions of jerk system and its various parameters under different conditions is important. In this article, we have studied phenomena of the Hopf bifurcation and chaos occurring in fractional ordered commensurate and incommensurate quadratic jerk system. The equilibrium points of the system are obtained and are found to be $(\pm\epsilon,0,0)$, where $\epsilon$ denotes the system parameter versus which bifurcation is analyzed. We have presented the criteria for commensurate and incommensurate quadratic jerk system to undergo a Hopf bifurcation. The value of  $\epsilon$ at which system undergoes Hopf bifurcation $\epsilon_{H}$ is obtained for both commensurate as well as incommensurate system. It is known that supercritical Hopf bifurcation leads to chaos. The obtained results are verified through numerical simulations versus the fractional order $\alpha$ and parameter $\epsilon$ and the explicit range in which the system exhibits chaos is found. A number of phase portraits, bifurcation diagrams, and Lyapunov exponent diagrams are presented to affirm the obtained chaotic range of parameters.
	\end{abstract}
	\section{Introduction}
	Fractional Calculus is a section of mathematics that involves the study of differential and integral operators of arbitrary real order\cite{podlubny1999fractional}. Fractional ordered dynamical systems (FODS) are recently in focus due to various qualitative and quantitative results exhibited  by them\cite{deshpande2019analysis,daftardar2010chaos,kaslik2012nonlinear}. There is a conjecture that for FODS with continuous time, there exists an order below which the system loses chaos\cite{DESHPANDE2017119}. FODS have vital applications in electrochemistry\cite{zou2018review},nuclear and particle physics\cite{herrmann2023fractional}, cryptographical models\cite{natiq2022image}, viscoelasticity\cite{craiem2010fractional}, control theory\cite{redhwan2019some}, meteorology\cite{yajima2018geometric} and neural network\cite{viera2022artificial}.
	
	The Jerk system, which is a quadratic system, has applications in numerous fields, one of the few are, chemical kinetics\cite{yadav2017stability}, seismic control of civil structures\cite{geoffrey2003quadratic}, electronic circuit implementation\cite{hammouch2018circuit}, encryption of image to audio data\cite{rajagopal2020exponential}. Chase \textit{et al.}\cite{geoffrey2003quadratic} have introduced a new method to improve the performance of civil structures undergoing large seismic events by regulating the total structural jerk. They have used a single Riccati equation to define the control gains by using quadractic regulator for total structral jerk. They have shown that quadratic jerk regulation responds better than the current structural control methods for seismic events. Rajagopal \textit{et al.} \cite{rajagopal2020exponential} have used the randomness of the jerk system to create an encryption algorithm that converts images into audio data and the image is recovered from the hidden audio file. They have designed a fractional order chaotic jerk system with one-quadratic nonlinearity and two-cubic nonlinearities and the system is solved by using Adomain decomposition method to generate complex chaotic signals. Kacar \textit{et al.} \cite{vaidyanathan2018new} have discovered the propertries of a new jerk system having two quadratic nonlinearities. They have shown that the new jerk system has complex dynamics characteristics which can be used for random number generator,image encryption algorithm, electronic
	circuit implementation, artificial neural network, genetic algorithm etc.
	From the survey above it is quite evident that the applications of jerk system has no bounds and hence the study of jerk system for various parameters under different conditions will create more efficient applications that can be rendered in real life.
	Hence the study of Quadractic jerk system is conducted in the next sections.
	
	\section{Preliminaries}\label{sec1}
	
	\begin{defn}[\cite{podlubny1999fractional}]
		The fractional integral of order $\alpha > 0$ of a real valued function $f$ is defined as
		\begin{align*}
			I^{\alpha}f(t)=\frac{1}{\Gamma(\alpha)}\int\limits_{0}^{t}\frac{f(\tau)}{(t-\tau)^{1-\alpha}}d\tau.
		\end{align*}
	\end{defn}
	
	\section{Analysis of Quadratic Jerk system}
	\begin{thm}
		Consider the system for $t>0$,\label{main1}
		\begin{equation}
			\begin{split}
				D^{\alpha}x(t)& = y(t),
				\\	D^{\alpha}y(t)& = z(t),
				\\D^{\alpha}z(t) &= -\epsilon^{2}-by(t)-a \epsilon z(t)+(x(t))^{2},
			\end{split}	
		\end{equation}
		where $a>0, b>0$, $\epsilon \in \mathbb{R}$ and $ 0<\alpha\leq1$ is the fractional order. Then a Hopf critical value of the bifurcation parameter $\epsilon$, denoted as $\epsilon_{H}$ is given as,
		\begin{equation}
			\epsilon_{H}=\frac{-(\gamma_{H}^{3}\cos3(\theta)+\gamma_{H}b\cos(\theta))}{(\gamma_{H}^{2}a\cos2(\theta)\mp2)},
		\end{equation}
		where $\gamma_{H}$ denotes the modulus of the eigen values of the $Jacobian$ at equilibrium points $E_{1}=(\epsilon,0,0), E_{2}=(-\epsilon,0,0)$ and $\alpha \neq\pm\frac{1}{\pi}\cos^{-1}\left( \pm\frac{2}{a\gamma_{H}^{2}}\right)$.
		
		\noindent This system will under go Hopf Bifurcation in following cases,\\
		\textbf{Case I} : If $\alpha>\frac{2}{3}$ and $b>0$,\\
		\textbf{Case II} : If $\alpha<\frac{2}{3}$ and $b<0$.

	\end{thm}
	\begin{proof}
		Consider the given system
		\begin{equation}
			\begin{split}\label{l1}
				D^{\alpha}x(t)& = y(t),
				\\	D^{\alpha}y(t)& = z(t),
				\\D^{\alpha}z(t) &= -\epsilon^{2}-by(t)-a \epsilon z(t)+(x(t))^{2}.
			\end{split}	
		\end{equation}
		To obtain equilibrium points, let,$y(t)=0$, $z(t)=0$ in the given system,
		\begin{equation}
			\begin{split}
				-\epsilon^{2}-by-a \epsilon z+x^{2}=0,\\x=\pm \epsilon.
			\end{split}
		\end{equation}
		Hence equilibrium points are
		$E_{1}=(\epsilon,0,0)$ and $E_{2}=(-\epsilon,0,0)$ .
		\\Let	$ x(t)=f_{1}, y(t)=f_{2}, z(t)=f_{3}$ in equation \eqref{l1}.\\
		The Jacobian Matrix is given by,
		\begin{equation}\label{l2}
			\begin{split}
				J_{E_{{1},{2}}}&=
				\begin{bmatrix}
					0 & 1 & 0
					\\ 0 & 0 & 1
					\\2x & -b & -a \epsilon    	
				\end{bmatrix}
				\\ c(\lambda)&= \lambda^{3}+a \epsilon \lambda^{2}+ \lambda(b+0+0)\mp 2 \epsilon\\&=\lambda^{3}+a \epsilon \lambda^{2}+ \lambda b\mp 2 \epsilon.
			\end{split}
		\end{equation}
		Let $\lambda=re^{i(\theta)}$, where $(\theta)=\pm \frac{\pi\alpha}{2}$.
		Substitute the value of $\lambda$ in equation \eqref{l2},
		$c(re^{i(\theta)})=(re^{i(\theta)})^{3}+a \epsilon( re^{i(\theta)})^{2}+ (re^{i (\theta)} )b\mp 2 \epsilon$ \\
		The real part of above equation is, $r^{3}\cos3(\theta)+r^{2}a\epsilon\cos2(\theta)+br\cos(\theta) \mp2\epsilon=0$,\\
		The imaginary part of above equation is, $r^{3}\sin3(\theta)+r^{2}a\epsilon\sin2(\theta)+br\sin(\theta) =0,$
		\begin{equation} \label{r12}
			r_{1,2}=\frac{-a \epsilon\sin2(\theta)\pm\sqrt{a^{2}\epsilon^{2}\sin^{2}2(\theta)-4 b\sin3(\theta)\sin(\theta)}}{2\sin3(\theta)} =\frac{-a \epsilon\sin2(\theta)\pm\sqrt{\Delta}}{2\sin3(\theta)},
		\end{equation} where $\Delta=a^{2}\epsilon^{2}\sin^{2}2(\theta)-4 b\sin3(\theta)\sin(\theta)$.\\
		Let, $\Delta > 0$ in equation \eqref{r12}, where $\Delta=a^{2}\epsilon^{2}\sin^{2}2(\theta)-4 b\sin3(\theta)\sin(\theta)$.
		
		\noindent We have, $(\theta)=\pm \frac{\pi\alpha}{2}$, then consider the following two cases,
		
		\noindent\textbf{Case I:} If $\alpha>\frac{2}{3}$. We obtain the following,
		\begin{table}[H]
			\centering
			\begin{tabular}{ | m{6em} | m{6em}| } 
				\hline
				$(\theta)>\frac{\pi}{3} $& $ (\theta)<-\frac{\pi}{3}   $  \\ 
				\hline
				$ \sin(\theta)>0 $& $ \sin(\theta)<0  $   \\ 
				\hline
				$\sin2(\theta)>0 $& $ \sin2(\theta)<0 $   \\ 
				\hline
				$\sin3(\theta)<0  $&  $\sin3(\theta)>0   $  \\ 
				\hline
			\end{tabular}
			\caption{Values of $\sin(\theta), \sin2(\theta), \sin3(\theta)$ for $(\theta)>\frac{\pi}{3} $ and $ (\theta)<-\frac{\pi}{3}.$}
		\end{table}
		\noindent Substitute the conditions in Table 1 in equation \eqref{r12}, the result is, $\sin(\theta)\sin3(\theta)<0$.
		Hence if $b>0$ then $\Delta>0$.\\
		
		\noindent\textbf{Case II:} If $\alpha<\frac{2}{3}$. We obtain the following,
		\begin{table}[H]
			\centering
			\begin{tabular}{ | m{6em} | m{6em}| } 
				\hline
				$ (\theta)<\frac{\pi}{3}$  \\ 
				\hline
				$ \sin(\theta)>0 $  \\ 
				\hline
				$\sin2(\theta)>0 $  \\ 
				\hline
				$\sin3(\theta)>0  $ \\ 
				\hline
			\end{tabular}
			\caption{Values of $\sin(\theta), \sin2(\theta), \sin3(\theta)$ for $(\theta)>\frac{\pi}{3}.$}
		\end{table}
		\noindent Substitute the conditions in Table 2 in equation \eqref{r12}, the result is,$\sin(\theta)\sin3(\theta)>0$.
		Hence if $b<0$ then $\Delta>0$.\\
		From equation \eqref{r12},
		\begin{equation}
			r_{1}r_{2}=\frac{b\sin(\theta)}{\sin3(\theta)}<0.
		\end{equation}
		From Case I and Case II it follows that either $r_{1}$ or $r_{2}$ will be positive.
		\\
		Let $\gamma_{H}$ denote the positive root and $\lambda_{1}=\gamma_{H}e^{\frac{i\alpha\pi}{2}}$, $\lambda_{2}=\gamma_{H}e^{-\frac{i\alpha\pi}{2}}$.
		Substitute the above values in $r^{3}\cos3(\theta)+r^{2}a\epsilon\cos2(\theta)+br\cos(\theta) \mp2\epsilon=0$ to find the Hopf Critical Value.
		Now $b,\gamma_{H},a$ and $(\theta)$ are known. We will solve to find the value of $\epsilon$.
		\begin{equation}
			\begin{split}
				&\gamma_{H}^{3}\cos3(\theta)+\gamma_{H}^{2}a\epsilon\cos2(\theta)+\gamma_{H}b\cos(\theta) \mp2\epsilon=0,
				\\&\gamma_{H}^{3}\cos3(\theta)+\gamma_{H}b\cos(\theta) +\epsilon(\gamma_{H}^{2}a\cos2(\theta)\mp2)=0,
				\\&\epsilon_{H}=\frac{-(\gamma_{H}^{3}\cos3(\theta)+\gamma_{H}b\cos(\theta))}{\gamma_{H}^{2}a\cos2(\theta)\mp2},
			\end{split}
		\end{equation}
		where $\gamma_{H}^{2}a\cos2(\theta)\mp2$ is non-zero.
		
		\begin{equation}
			\begin{split}
				\gamma_{H}^{2}a\cos2(\theta)\mp2\neq0
				\\\alpha \neq\pm\frac{1}{\pi}\cos^{-1}\left( \pm\frac{2}{a\gamma_{H}^{2}}\right).
			\end{split}
		\end{equation}
		Hence the proof.
	\end{proof}

	\begin{figure}[p]
		\centering
		\begin{subfigure}[b]{0.45\textwidth}
			\centering
			\includegraphics[width=\textwidth]{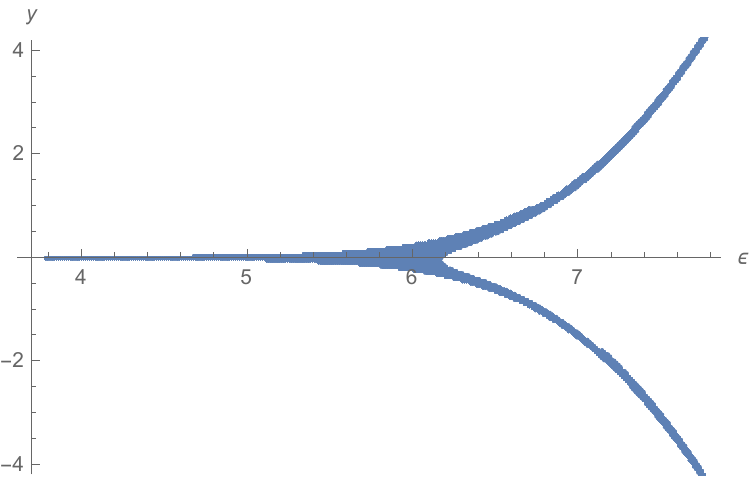}
			\subcaption{Bifurcation diagram with parameter $\alpha$=0.91, $\epsilon$ from range 3.781 to 7.780, initial conditions (0,0,0)}
			\label{digg1}
		\end{subfigure}
		\begin{subfigure}[b]{0.45\textwidth}
			\centering
			\includegraphics[width=\textwidth]{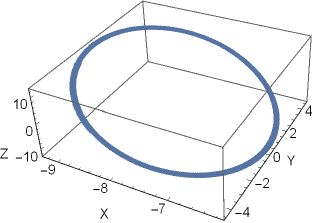}
			\subcaption{Phase Potrait for $\epsilon$=7.780 and $\alpha$=0.91}
			\label{digg2}
		\end{subfigure}
		\caption{}
	\end{figure}
	
	\begin{figure}[p]
		\centering
		\includegraphics[width=0.7\linewidth]{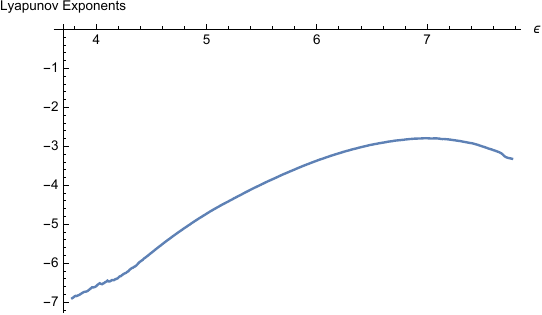}
		\caption{Lypapunov diagram with parameter $\epsilon$ from range 3.781 to 7.780 and lyapunov exponents}
		\label{fig:PhaseP_epsilon7780}
	\end{figure}

	\begin{figure}[p]
		\centering
		\begin{subfigure}[b]{0.45\textwidth}
			\centering
			\includegraphics[width=\textwidth]{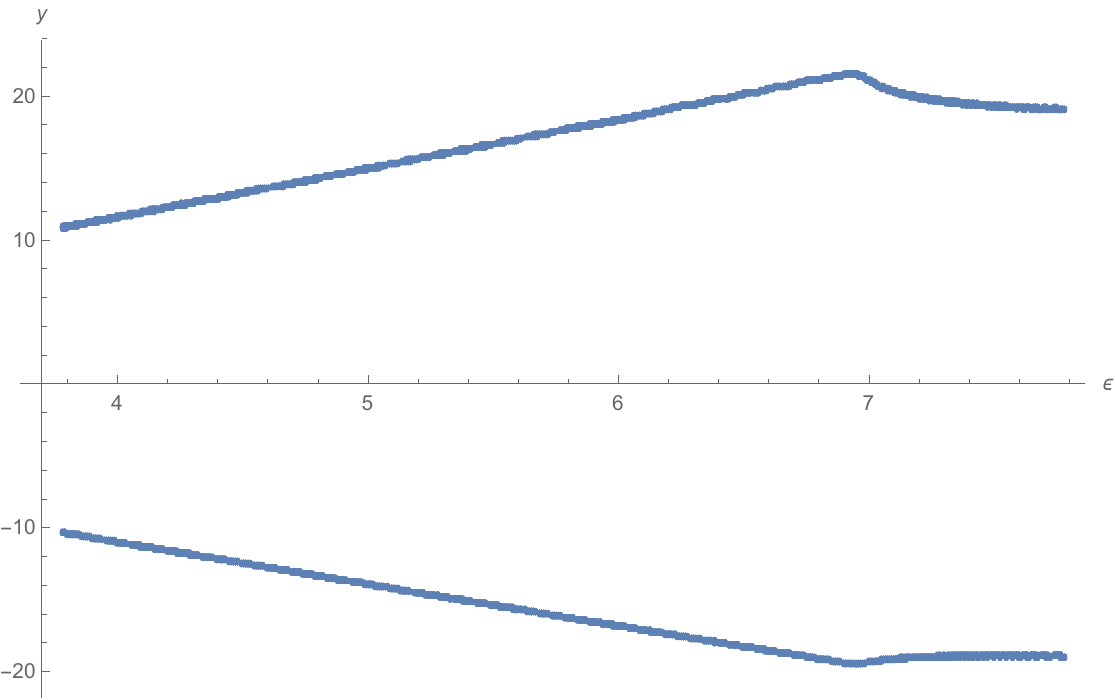}
			\subcaption{Bifurcation diagram with parameter $\alpha$=0.98, $\epsilon$ from range 3.781 to 7.780, initial conditions (0,0,0)}
			\label{figgg1}
		\end{subfigure}
		\begin{subfigure}[b]{0.45\textwidth}
			\centering
			\includegraphics[width=\textwidth]{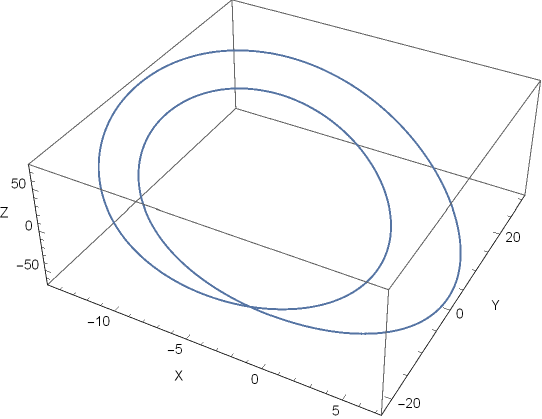}
			\subcaption{Phase Potrait for $\epsilon$=7.750 and $\alpha$=0.98}
			\label{figgg2}
		\end{subfigure}
		\caption{}
		\label{}
	\end{figure}
	
	\begin{figure}[p]
		\centering
		\includegraphics[width=0.7\linewidth]{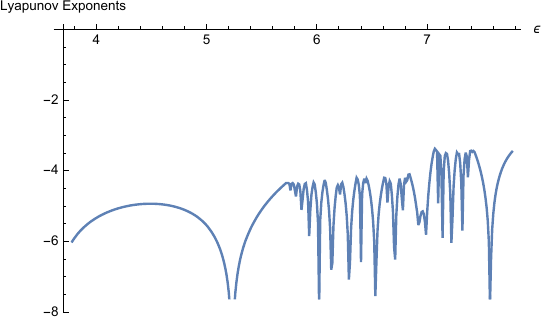}
		\caption{Lypapunov diagram with parameter $\epsilon$ from range 3.781 to 7.780 and lyapunov exponents}
		\label{fig:091phasepepsilon7750}
	\end{figure}
	
	\begin{figure}[p]
		\centering
		\begin{subfigure}[b]{0.45\textwidth}
			\centering
			\includegraphics[width=\textwidth]{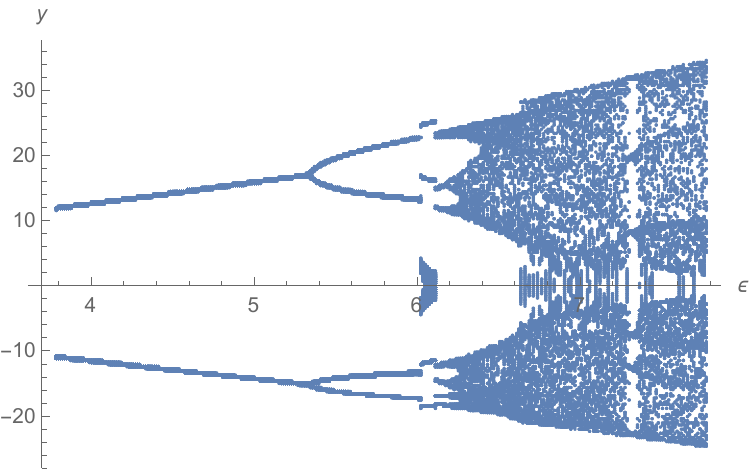}
			\subcaption{Bifurcation diagram with parameter $\alpha$=0.99, $\epsilon$ from range 3.781 to 7.780, initial conditions (0,0,0)}
			\label{ex1}
		\end{subfigure}
		\begin{subfigure}[b]{0.45\textwidth}
			\centering
			\includegraphics[width=\textwidth]{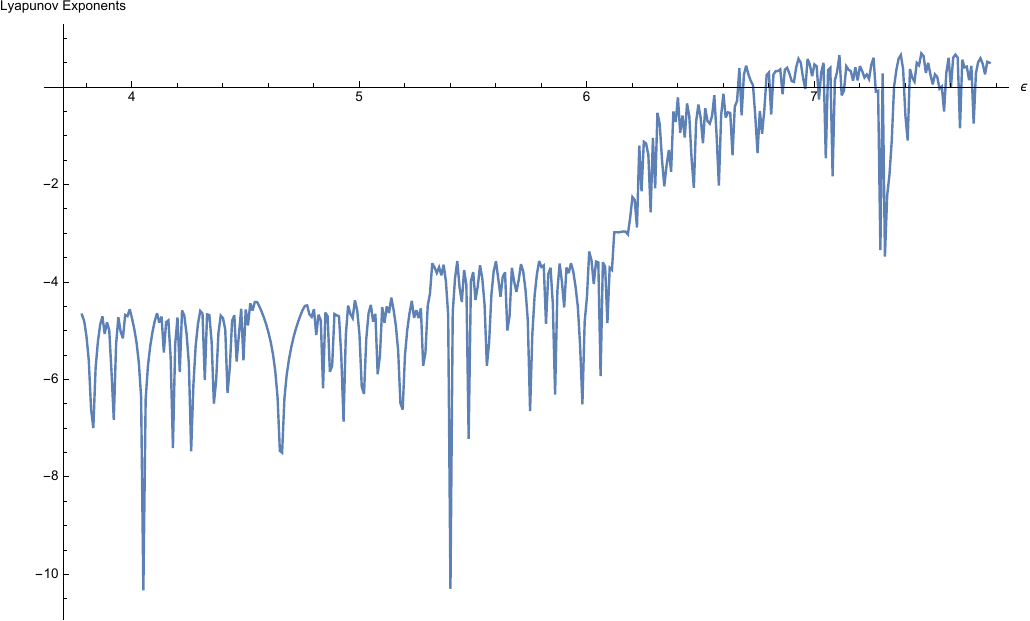}
			\subcaption{Lypapunov diagram with parameter $\epsilon$ from range 3.781 to 7.780 and lyapunov exponents}
				\label{ex2}
		\end{subfigure}
		\caption{}
	\end{figure}
	
	\begin{figure}[p]
		\centering
		\begin{subfigure}[b]{0.45\textwidth}
			\includegraphics[width=0.7\linewidth]{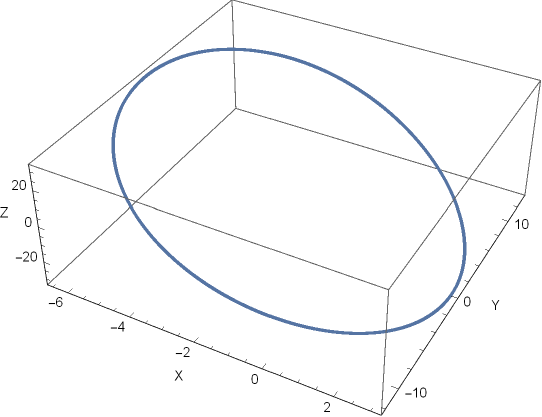}
			\caption{Phase Potrait for $\epsilon$=3.783 and $\alpha$=0.99}
			\label{dig1}
		\end{subfigure}
		\begin{subfigure}[b]{0.45\textwidth}
			\includegraphics[width=0.7\linewidth]{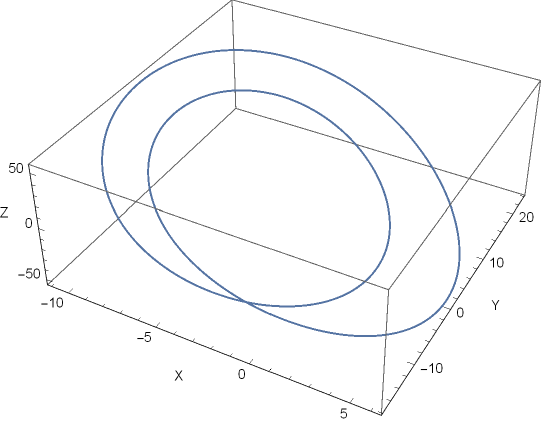}
			\caption{Phase Potrait for $\epsilon$=5.781 and $\alpha$=0.99}
			\label{dig2}
		\end{subfigure}
		\begin{subfigure}[b]{0.45\textwidth}
			\includegraphics[width=0.7\linewidth]{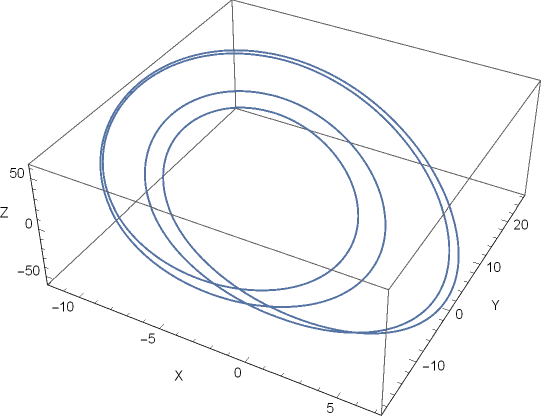}
			\caption{Phase Potrait for $\epsilon$=6.142 and $\alpha$=0.99}
			\label{dig3}
		\end{subfigure}
		\begin{subfigure}[b]{0.45\textwidth}
			\includegraphics[width=0.7\linewidth]{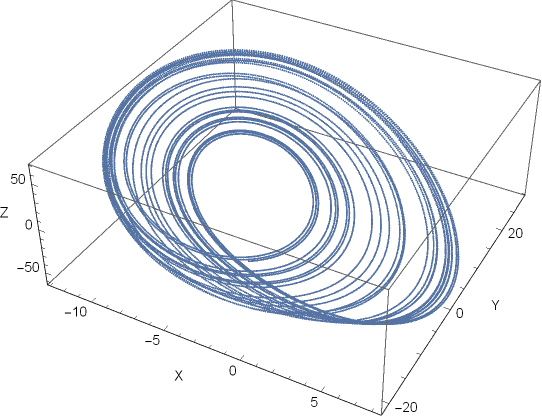}
			\caption{Phase Potrait for $\epsilon$=6.611 and $\alpha$=0.99}
			\label{dig4}
		\end{subfigure}
		\begin{subfigure}[b]{0.45\textwidth}
			\includegraphics[width=0.7\linewidth]{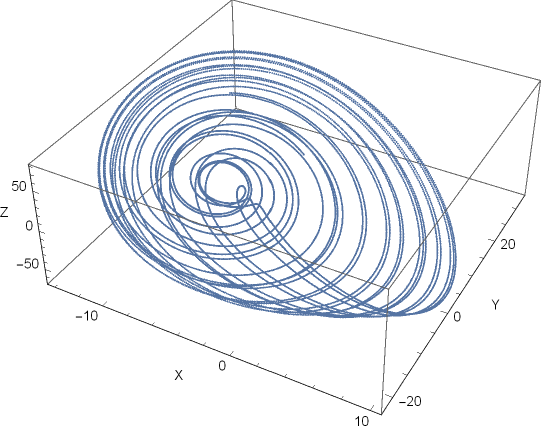}
			\caption{Phase Potrait for $\epsilon$=7.079 and $\alpha$=0.99}
			\label{dig5}
		\end{subfigure}
		\begin{subfigure}[b]{0.45\textwidth}
			\includegraphics[width=0.7\linewidth]{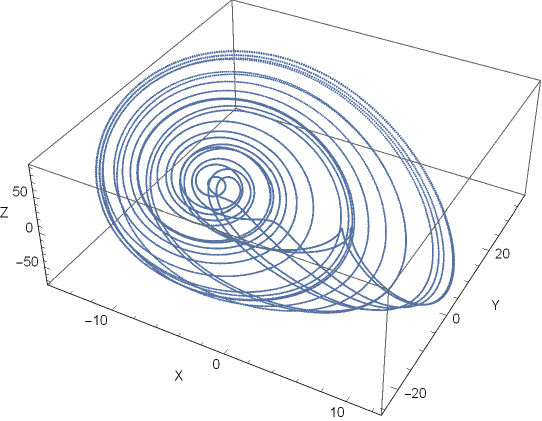}
			\caption{Phase Potrait for $\epsilon$=7.780 and $\alpha$=0.99}
			\label{dig6}
		\end{subfigure}
		\caption{}
		\label{}
	\end{figure}
	
	\section{Result Analysis of Commensurate Quadratic Jerk system }
	
	In this section we have analysied system for \eqref{main1}. Figure \eqref{digg1} represents a process known as bifurcation for parameters $\alpha=0.91$ and $\epsilon$ from range $3.781$ to $7.780$ with $a=0.129, b=7$. Two branches in the diagram shows data two distinct paths traced by data. Figure \eqref{digg2} represents phase potrait for $\epsilon=7.780$, $\alpha=0.91$. The single orbit concludes that there is no chaos exhibited at this point.
	
	\noindent Figure \eqref{fig:PhaseP_epsilon7780} represents the distribution of lyapunov exponents with parameter $\epsilon$ from range $3.781$ to $7.780$.
	
	\noindent Figure \eqref{figgg1} represents a process known as bifurcation for parameters $\alpha=0.98$ and $\epsilon$ from range $3.781$ to $7.780$ with $a=0.129, b=7$. Figure \eqref{figgg2} represents phase potrait for $\epsilon=7.780$, $\alpha=0.98$. Two orbits conclude that there is no chaos exhibited at this point but it may exist in the further analysis.
	
	\noindent Figure \eqref{fig:091phasepepsilon7750} represents the distribution of lyapunov exponents with parameter $\epsilon$ from range $3.781$ to $7.780$, the negative values of lyapunov exponents confirm that there is no chaos.
	
	\noindent Figure \eqref{ex1} represents a process known as bifurcation for parameters $\alpha=0.99$ and $\epsilon$ from range $3.781$ to $7.780$ with $a=0.129, b=7$. The dense distribution of the data represents exponential increase of orbits leading to chaos. Figure \eqref{ex2} represents the distribution of lyapunov exponents with parameter $\epsilon$ from range $3.781$ to $7.780$. The positive values for lyapunov exponents in the later part confirms presence of chaos.
	
	\noindent Figure \eqref{dig1} represents phase potrait at $3.783$ and $\alpha=0.99$. Single orbit represents no chaos. Figure \eqref{dig2} and \eqref{dig3} show the increase in number of orbits which will lead to chaos. Figure \eqref{dig4} and \eqref{dig5} shows the begining of chaos. Figure \eqref{dig6} shows the chaos for the system \eqref{main1} at parameters $\epsilon=7.780$, $\alpha=0.99$, $a=0.129$ and $b=7$.

	\section{Analysis of Incommensurate Quadratic Jerk system}
	\begin{thm}
		Consider the system for $t>0$,\label{main2}
		\begin{equation}
			\begin{split}
				D^{\alpha_{1}}x(t)& = y(t),
				\\	D^{\alpha_{2}}y(t)& = z(t),
				\\D^{\alpha_{3}}z(t) &= -\epsilon^{2}-by(t)-a \epsilon z(t)+(x(t))^{2},
			\end{split}	
		\end{equation}
		where $a>0, b>0$, $\epsilon \in \mathbb{R}$ and $ 0<\alpha_{1},\alpha_{2},\alpha_{3}\leq1$ is the fractional order. 
		In above equation $\alpha_{1}=\frac{v_{1}}{u_{1}}, \alpha_{2}=\frac{v_{2}}{u_{2}}, \alpha_{3}=\frac{v_{3}}{u_{3}}$ and $u_{1}, u_{2}, u_{3}, v_{1}, v_{2}, v_{3}\in \mathbb{Z}$, $gcd(u_{i}, v_{i})=1, M=lcm(u_{1}, u_{2}, u_{3}).$
		
		This system will exhibit Hopf bifurcation in following cases for $a=0.129$ and $b=7$:

		\noindent \textbf{Case I} : If $p\geq m$,\\
		\textbf{Case II} : If $m>p$ and $0<(p+q+m)\theta\leq\pi$.
		
		A Hopf critical value of the bifurcation parameter $\epsilon$ in each case, denoted as $\epsilon_{H}$ is given as,
		\begin{equation}
			\epsilon_{H}=\frac{-(\gamma_{H}^{p+q+m}\cos(p+q+m)(\theta)+\gamma_{H}^{p}\cos p(\theta)b)}{\gamma_{H}^{p+q}\cos(p+q)(\theta)a\mp2},
		\end{equation}	
		where $\gamma_{H}$ denotes the modulus of the eigen values of the $Jacobian$ at equilibrium points $E_{1}=(\epsilon,0,0), E_{2}=(-\epsilon,0,0)$ of the system.

	\end{thm}
	\begin{proof}
		Consider the given system
		\begin{equation}
			\begin{split}\label{l1}
				D^{\alpha_{1}}x(t)& = y(t),
				\\	D^{\alpha_{2}}y(t)& = z(t),
				\\D^{\alpha_{3}}z(t) &= -\epsilon^{2}-by(t)-a \epsilon z(t)+(x(t))^{2}.
			\end{split}	
		\end{equation}
		To obtain equilibrium points, let,$y(t)=0$, $z(t)=0$ in the given system,
		\begin{equation}
			\begin{split}
				-\epsilon^{2}-by-a \epsilon z+x^{2}=0,\\x=\pm \epsilon.
			\end{split}
		\end{equation}
		Hence equilibrium points are
		$E_{1}=(\epsilon,0,0)$ and $E_{2}=(-\epsilon,0,0)$ .
		
	\noindent	Let	$ x(t)=f_{1}, y(t)=f_{2}, z(t)=f_{3}$ in equation \eqref{l1}.\\
		The Jacobian Matrix is given by,
		\begin{equation}\label{l2}
			\begin{split}
				J_{E_{{1},{2}}}&=
				\begin{bmatrix}
					0 & 1 & 0
					\\ 0 & 0 & 1
					\\2x & -b & -a \epsilon    	
				\end{bmatrix}
				\\ &=\begin{bmatrix}
					0 & 1 & 0
					\\ 0 & 0 & 1
					\\\pm 2\epsilon & -b & -a \epsilon    	
				\end{bmatrix}
			\end{split}
		\end{equation}
		
		\noindent The characteristic polynomial is given by,
		\begin{equation}\label{l2}
			\begin{split}
				c(\lambda)&=
				det\left(\begin{bmatrix}
					\lambda^{M\alpha_{1}} & 0 & 0
					\\ 0 & \lambda^{M\alpha_{2}} & 0
					\\0 & 0 & \lambda^{M\alpha_{3}} \epsilon    	
				\end{bmatrix}-\begin{bmatrix}
					0 & 1 & 0
					\\ 0 & 0 & 1
					\\\pm 2\epsilon & -b & -a \epsilon    	
				\end{bmatrix}\right)
				\\ &=\begin{bmatrix}
					\lambda^{M\alpha_{1}}& -1 & 0
					\\ 0 & \lambda^{M\alpha_{2}} & -1
					\\\mp 2\epsilon & b & \lambda^{M\alpha_{3}} +a \epsilon   	
				\end{bmatrix}.
				\\&=\lambda^{M\alpha_{1}+M\alpha_{2}+M\alpha_{3}}+\lambda^{M\alpha_{1}+M\alpha_{2}} a \epsilon+\lambda^{M\alpha_{1}} b\pm2\epsilon=0
			\end{split}
		\end{equation}
		Now $|arg(\lambda)|=\frac{\pi}{2M}$.
		
		\noindent Let $\lambda=re^{i(\theta)}$ where $(\theta)=\pm\frac{\pi}{2M}$.
		\begin{equation}
		c(re^{i(\theta)})=(re^{i(\theta)})^{M\alpha_{1}+M\alpha_{2}+M\alpha_{3}}+(re^{i(\theta)})^{M\alpha_{1}+M\alpha_{2}} a \epsilon+(re^{i(\theta)})a^{M\alpha_{1}} b\pm2\epsilon=0
	    \end{equation}
\noindent We have $e^{i\theta}=\cos(\theta)+i\sin(\theta)$
		
	\noindent Let $M\alpha_{1}=p, M\alpha_{2}=q, M\alpha_{3}=m$.
		\begin{equation}\label{aa1}
			\begin{split}
				c(r(\cos(\theta)+i\sin(\theta)))=(r(\cos(\theta)+i\sin(\theta)))^{p+q+m}\\+(r(\cos(\theta)+i\sin(\theta)))^{p+q}a\epsilon+((\cos(\theta)+i\sin(\theta))^{p} b\pm2\epsilon=0
			\end{split}
		\end{equation}
		
		\noindent Equating real and imaginay parts of \eqref{aa1} we get,
		
		\noindent Real part of \eqref{aa1},
		\begin{equation}\label{aa2}
			r^{p+q+m}\cos(p+q+m)(\theta)+r^{p+q}\cos(p+q)(\theta)a\epsilon+r^{p}\cos(p)(\theta)b\pm2\epsilon=0
		\end{equation}
		Imaginary part of\eqref{aa1},
		\begin{equation}\label{aa3}
			r^{p+q+m}\sin(p+q+m)(\theta)+r^{p+q}\sin(p+q)(\theta)a\epsilon+r^{p}\sin(p)(\theta)b=0
		\end{equation}
		
		Let $\epsilon_{H}$ be the bifurcation parameter. Hence from \eqref{aa2} we have,
		\begin{equation}
			\epsilon_{H}=\frac{-(\gamma_{H}^{p+q+m}\cos(p+q+m)(\theta)+\gamma_{H}^{p}\cos p(\theta)b)}{\gamma_{H}^{p+q}\cos(p+q)(\theta)a\mp2},
		\end{equation}	
		where $\gamma_{H}$ denotes the modulus of the eigen values of the $Jacobian$ at equilibrium points $E_{1}=(\epsilon,0,0),$ 
		
		$E_{2}=(-\epsilon,0,0)$ of the system.
		
		Consider \eqref{aa3} and substitute $\epsilon=\epsilon_{H}$.
		\begin{equation}
			\begin{split}
				r^{p+q+m}\sin(p+q+m)(\theta)+r^{p}\sin(p)(\theta)b+\\r^{p+q}\sin(p+q)(\theta)a\left( \frac{-(\gamma_{H}^{p+q+m}\cos(p+q+m)(\theta)+\gamma_{H}^{p}\cos p(\theta)b)}{\gamma_{H}^{p+q}\cos(p+q)(\theta)a\mp2}\right)=0
			\end{split}
		\end{equation}
		\begin{equation}
			\begin{split}
				r^{2p+2q+m}s\sin (m\theta)-r^{2p+q}ab\sin(q\theta)+r^{p+q+m}(\pm2\sin(p+q+m)(\theta))+\\r^{p}(\pm2\sin(p\theta)b)=0
			\end{split}
		\end{equation}
		
		We will consider the following equation,
		\begin{equation}\label{aa4}
			r^{2p+2q+m}\sin (m\theta)-r^{2p+q}ab\sin(q\theta)-2 r^{p+q+m}\sin(p+q+m)(\theta))-2br^{p}\sin(p\theta)=0
		\end{equation}
		
		We apply Routh-Hourwiz criteria to equation \eqref{aa4} to check if it has real positive root.
		
		\textbf{Case I}: $p>m$ i.e. $(r^{2p+q}>r^{p+q+m})$
		
		Hence we have,
		
		\begin{tabular}{cc}
			$a\sin(m\theta)$&$-2\sin(p+q+m)(\theta) $ \\\\
			$-ab\sin(q\theta)$ &$-2b\sin(p\theta)$ \\\\
			$\frac{-2\sin(m+q)(\theta)\sin(p+q)(\theta)}{\sin(q\theta)}$  \\
		\end{tabular}
		
		Consider $a\sin(m\theta)$,
		
		$0<(m\theta)<m\frac{\pi}{2M}=\alpha_{3}\frac{\pi}{2}<\frac{\pi}{2}$.Hence $(m\theta)\in \left(0,\frac{\pi}{2} \right)$,$a\sin(m\theta)>0$.
		
		Consider $-ab\sin(q\theta)$,
		
		$0<(q\theta)<q\frac{\pi}{2M}=\alpha_{2}\frac{\pi}{2}<\frac{\pi}{2}$, $(q\theta)\in \left(0,\frac{\pi}{2} \right)$. Now $a>0, b>0,\sin(q\theta)$. Hence $-ab\sin(q\theta)<0.$
		
		Consider $\frac{-2\sin(m+q)(\theta)\sin(p+q)(\theta)}{\sin(q\theta)}$,
		
		$0<((m+q)\theta)<(m+q)\frac{\pi}{2M}=(\alpha_{2}+\alpha_{3})\frac{\pi}{2}<\pi, (m+q)(\theta)\in(0,\pi)$. Hence $\sin(m+q)(\theta)>0$.
		
		$0<((p+q)\theta)<(p+q)\frac{\pi}{2M}=(\alpha_{1}+\alpha_{2})\frac{\pi}{2}<\pi, (p+q)(\theta)\in(0,\pi)$. Hence $\sin(p+q)(\theta)>0$. Also $\sin(q\theta)$.
		
		Hence $\frac{-2\sin(m+q)(\theta)\sin(p+q)(\theta)}{\sin(q\theta)}<0$.
		
		Since one sign inversion occurs in Case I by Routh-Hourwiz criteria one positive root $\gamma_{H}$ shall occur.
		
		\textbf{Case II}: $p<m$ i.e. $(r^{2p+q}<r^{p+q+m})$
		
		Hence we have,\\
		\begin{tabular}{cc}
			$a\sin(m)(\theta)$&$-ab\sin(q)(\theta) $ \\\\
			$-2\sin(p+q+m)(\theta)$ &$-2b\sin(p)(\theta)$ \\\\
			$-ab\frac{\sin(m+q)(\theta)\sin(p+q)(\theta)}{\sin(p+q+m)(\theta)}$\\
		\end{tabular}
		
		Consider $a\sin(m\theta)$,
		
		$0<(m\theta)<m\frac{\pi}{2M}=\alpha_{3}\frac{\pi}{2}<\frac{\pi}{2}$.Hence $(m\theta)\in \left(0,\frac{\pi}{2} \right)$,$a\sin(m\theta)>0$.
		
		Consider $-2\sin(p+q+m)(\theta)$,
		
		$0<(p+q+m)(\theta)<(p+q+m)\frac{\pi}{2M}=(\alpha_{1}+\alpha_{2}+\alpha_{3})\frac{\pi}{2}<3\frac{\pi}{2}$
		
		Further we have three more cases for $(p+q+m)(\theta)$ lying between $0$ to $\frac{3\pi}{2}$.
		
		Case (i): $0<(p+q+m)(\theta)\leq\pi$ , hence $\sin(p+q+m)(\theta)>0$ then $-2\sin(p+q+m)(\theta)<0$.
		
		Case (ii): $\pi<(p+q+m)(\theta)\leq\frac{3\pi}{2}$,hence $\sin(p+q+m)(\theta)<0$ then $-2\sin(p+q+m)(\theta)>0$.
		
		Case (iii):$(p+q+m)(\theta)=\pi$,hence $\sin(p+q+m)(\theta)>0$ then $-2\sin(p+q+m)(\theta)<0$.
		
		Consider $-ab\frac{\sin(m+q)(\theta)\sin(p+q)(\theta)}{\sin(p+q+m)(\theta)}$,
		We have $a>0, b>0, \sin(p)(\theta)>0,$
		
		$\sin(p)(\theta)\sin(m)(\theta)>0$.
		If Case (i) then $-ab\frac{\sin(m+q)(\theta)\sin(p+q)(\theta)}{\sin(p+q+m)(\theta)}<0$,
		
		if Case (ii) then $-ab\frac{\sin(m+q)(\theta)\sin(p+q)(\theta)}{\sin(p+q+m)(\theta)}>0$, if Case (iii) then $-ab\frac{\sin(m+q)(\theta)\sin(p+q)(\theta)}{\sin(p+q+m)(\theta)}<0$.
		
		Since one sign inversion occurs in Case I and Case III by Routh-Hourwiz criteria one positive root $\gamma_{H}$ shall occur.
		
		Case II is not valid since there is no sign inversion, hence no positive root.
		
		\textbf{Case III}: $p=m$ i.e. $(r^{p}=r^{m})$
		
		The equation (\ref{aa4}) will be $r^{2p+2q}a\sin(p)(\theta)-r^{p+q}(ab\sin(q)(\theta)+2\sin(2p+q)(\theta))-2b\sin(p)(\theta)=0$.

		Hence we have,

		\begin{tabular}{cc}
			$a\sin(p)(\theta)$&$-2b\sin(p)(\theta) $ \\\\
			$-(ab\sin(q)(\theta)+2\sin(2p+q)(\theta))$ &0 \\\\
			$-2b\sin(p)(\theta)$\\
		\end{tabular}
		Consider $a\sin(p\theta)$,
		
		$0<(p\theta)<p\frac{\pi}{2M}=\alpha_{1}\frac{\pi}{2}<\frac{\pi}{2}$.Hence $(p\theta)\in \left(0,\frac{\pi}{2} \right)$,$a\sin(p\theta)>0$.

		Consider $-(ab\sin(q)(\theta)+2\sin(2p+q)(\theta))$,
		
		$0<(2p+q)(\theta)<(2p+q)\frac{\pi}{2M}=(2\alpha_{1}+\alpha_{2})\frac{\pi}{2}<3\frac{\pi}{2}$ . Also $\sin(q)(\theta)>0$.
		
		Further we have three more cases for $(2p+q)(\theta)$ lying between $0$ to $\frac{3\pi}{2}$.
		
		Case (i): $0<(2p+q)(\theta)\leq\pi$ , hence $-(ab\sin(q)(\theta)+2\sin(2p+q)(\theta))<0$.
		
		Case (ii): $\pi<(2p+q)(\theta)\leq\frac{3\pi}{2}$,hence $-(ab\sin(q)(\theta)+2\sin(2p+q)(\theta))>0$(For $a=0.129, b=7$).
		
		Case (iii):$(p+q+m)(\theta)=\pi$, hence $-(ab\sin(q)(\theta)+2\sin(2p+q)(\theta))<0$.
		
		Consider $-2b\sin(p\theta)$, $0<(p\theta)<p\frac{\pi}{2M}=\alpha_{1}\frac{\pi}{2}<\frac{\pi}{2}$.
		
		Hence $(p\theta)\in \left(0,\frac{\pi}{2} \right)$, $-2b\sin(p\theta)<0$.
		
		Since one sign inversion occurs in all Cases by Routh-Hourwiz criteria one positive root $\gamma_{H}$ shall occur.
	
		\begin{figure}[p]
		\centering
		\includegraphics[width=0.7\linewidth]{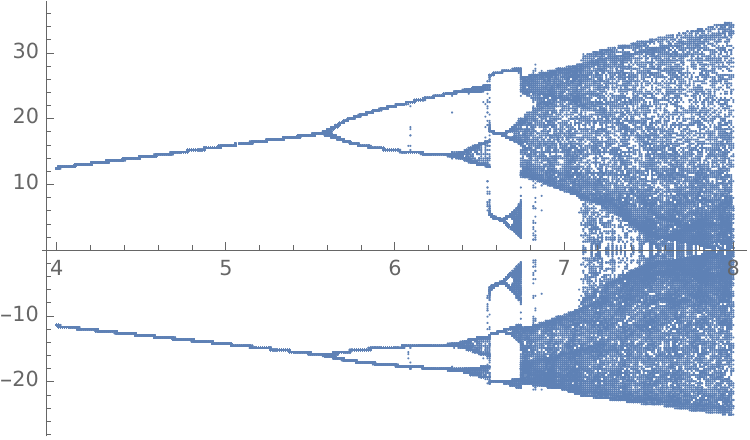}
		\caption{Bifurcation diagram with parameter $\epsilon$ from range 4.000 to 7.999 and $\alpha_{1}=1$, $\alpha_{2}=0.99$, $\alpha_{3}=1$.}
		\label{Bif_4000}
	\end{figure}
	
	\begin{figure}[p]
			\centering
			\begin{subfigure}[b]{0.45\textwidth}
					\includegraphics[width=0.7\linewidth]{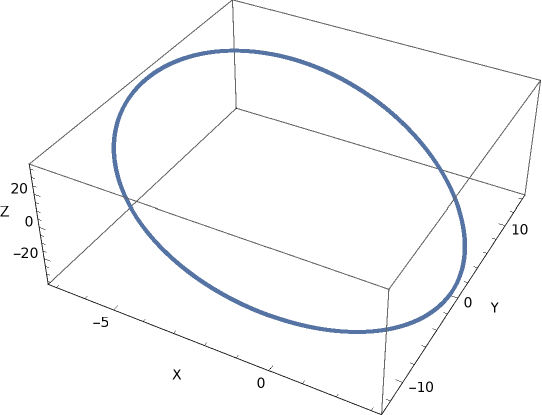}
					\caption{Phase Portrait for $\epsilon$=4.102 and $\alpha_{1}=1$, $\alpha_{2}=0.99$, $\alpha_{3}=1$}
					\label{ppp1}
				\end{subfigure}
			\begin{subfigure}[b]{0.45\textwidth}
					\includegraphics[width=0.7\linewidth]{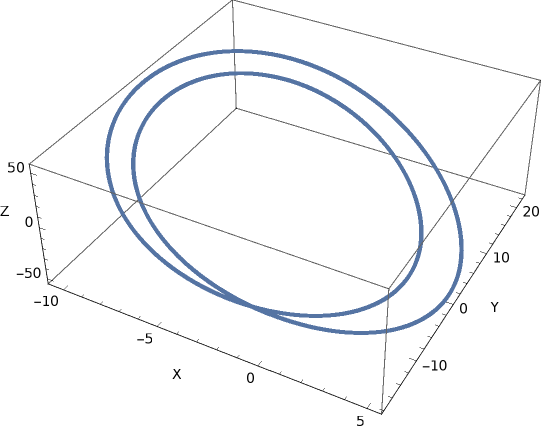}
					\caption{Phase Portrait for $\epsilon$=5.770 and $\alpha_{1}=1$, $\alpha_{2}=0.99$, $\alpha_{3}=1$}
					\label{ppp2}
				\end{subfigure}
			\begin{subfigure}[b]{0.45\textwidth}
					\includegraphics[width=0.7\linewidth]{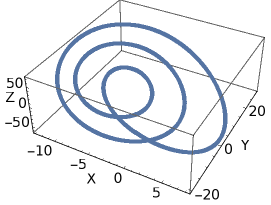}
					\caption{Phase Portrait for $\epsilon$=6.574 and $\alpha_{1}=1$, $\alpha_{2}=0.99$, $\alpha_{3}=1$}
					\label{ppp3}
				\end{subfigure}
			\begin{subfigure}[b]{0.45\textwidth}
					\includegraphics[width=0.7\linewidth]{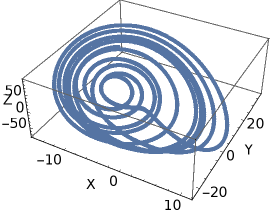}
					\caption{Phase Portrait for $\epsilon$=7.913 and $\alpha_{1}=1$, $\alpha_{2}=0.99$, $\alpha_{3}=1$}
					\label{ppp4}
				\end{subfigure}
			\caption{Phase Portraits with parameter $\epsilon$ from range 4.000 to 7.999 and $\alpha_{1}=1$, $\alpha_{2}=0.99$, $\alpha_{3}=1$.}
		\end{figure}

	\end{proof}

\section{Result Analysis of Incommensurate Quadratic Jerk system }

In this section we have analysied system for \eqref{main1}. Figure \eqref{Bif_4000} represents a process known as bifurcation for parameter $\epsilon$ from range 4.000 to 7.999 and $\alpha_{1}=1$, $\alpha_{2}=0.99$, $\alpha_{3}=1$. The dense distribution of the data represents exponential increase of orbits leading to chaos.

\noindent Figure \eqref{ppp1} represents phase potrait at $3.783$ and $\alpha=0.99$. Single orbit represents no chaos. Figure \eqref{ppp2} and \eqref{ppp3} show the increase in number of orbits which will lead to chaos. Figure \eqref{ppp4} shows the chaos for the system \eqref{main1} at parameters $\epsilon=7.913$, $\alpha_{1}=1$, $\alpha_{2}=0.99$, $\alpha_{3}=1$, $a=0.129$ and $b=7$.

	\section{Conclusion}
	In this article, we studied the chaos and bifurcations of the quadratic jerk system. The system in equation \eqref{main1} has been taken into consideration with parameters $a, b$ and $\epsilon$. The fractional order ($\alpha$) is $0<\alpha\leq1$. The numerical simulations haven been conducted to obtain two cases for which the system in equation \eqref{main1} goes under zero-hopf bifurcation. To support the numerical results, phase potraits, bifurcation and lyapunov digrams are presented. The conditions and parameter values have been identified for which the system in equation \eqref{main1}  is chaotic.

	\bibliographystyle{ieeetr}
	
	\bibliography{bibfile}
\end{document}